\documentclass[reqno]{amsart}
\usepackage{amssymb,latexsym,amsmath,amsthm,enumerate,amsbsy}
\usepackage[mathscr]{eucal}
\usepackage{framed,color,graphicx}
\usepackage{mathrsfs}
\usepackage[all]{xy}
\usepackage{tikz}
\usepackage{cite}
\usepackage{needspace}
\usepackage{url}

\makeatletter
\@namedef{subjclassname@2020}{\textup{2020} Mathematics Subject Classification}
\makeatother

\usetikzlibrary{positioning,decorations.pathreplacing,patterns,decorations.pathmorphing}

\tikzset{%
element/.style={draw, shape=circle, fill=white, inner sep=1.4pt}
}

\IfFileExists{ubbold.fd}{%
  \DeclareSymbolFont{bbold}{U}{bbold}{m}{n}
  \DeclareSymbolFontAlphabet{\mathbbold}{bbold}
}{%
  \newcommand{\mathbbold}{\mathbb}
}

\theoremstyle{plain}
\newtheorem{theorem}{Theorem}[section]
\newtheorem{lemma}[theorem]{Lemma}

\newtheorem{proposition}[theorem]{Proposition}

\newtheorem{remark}[theorem]{Remark}

\renewcommand{\ge}{\geqslant}
\renewcommand{\le}{\leqslant}

\newcommand{\bh}{\mathbf{h}}

\newcommand{\bp}{\mathbf{p}}
\newcommand{\bq}{\mathbf{q}}
\newcommand{\br}{\mathbf{r}}

\newcommand{\bu}{\mathbf{u}}
\newcommand{\bv}{\mathbf{v}}
\newcommand{\bw}{\mathbf{w}}
\newcommand{\bx}{\mathbf{x}}
\newcommand{\by}{\mathbf{y}}
\newcommand{\bz}{\mathbf{z}}

\newcommand{\Colorings}{\mathcal{C}}
\newcommand{\eps}{\varepsilon}

\newcommand{\dotcup}{\mathbin{\dot\cup}}
\newcommand{\rev}{\mathrm{rev}}
\newcommand{\op}{\mathrm{op}}
\newcommand{\intr}{\operatorname{int}}

\begin{document}

\title[Equational theories of flat semirings $S(W)$]
{Equational theories of flat semirings $S(W)$}

\author{Zidong Gao}
\address{School of Mathematics, Northwest University, Xi'an, 710127, Shaanxi, P.R. China}
\email{zidonggao@yeah.net}

\author{Yilin Zhou}
\address{School of Mathematics, Northwest University, Xi'an, 710127, Shaanxi, P.R. China}
\email{zhou990619@163.com}


\begin{abstract}
We study the finite basis problem for flat semirings of the form $S(W)$,
where $W$ is a nonempty set of nonempty words, possibly infinite.
We prove that $S(W)$ is finitely based if and only if the strict interior
of every word in $W$ is linear; here the strict interior of a word is
obtained by deleting its first and last letters. This completely resolves
an open problem raised by Jackson et al.~(J. Algebra 611: 211--245, 2022).
\end{abstract}

\maketitle
\section{Introduction}
An \emph{additively idempotent semiring} (or ai-semiring for short) is
an algebra $(S, +, \cdot)$ such that
\begin{itemize}
\item the additive reduct $(S, +)$ is a commutative idempotent semigroup;

\item the multiplicative reduct $(S, \cdot)$ is a semigroup;

\item the distributive laws hold:
\[
x(y+z)\approx xy+xz, \quad (x+y)z\approx xz+yz.
\]
\end{itemize}

The additive reduct of an ai-semiring is a semilattice,
and the order $\leq$ on an ai-semiring defined by
\[
a \leq b\;\Leftrightarrow\; a+b=b
\]
is compatible with both addition and multiplication.
Consequently, an ai-semiring is often called a \emph{semilattice-ordered semigroup} (see \cite{kp}).
Whenever an order on an ai-semiring is mentioned, it refers to this order.
For an ai-semiring $S$, if $J$ is both a multiplicative ideal and an additive order filter of $S$,
we may form the ideal quotient $S/J$ by collapsing all elements of $J$ to a single element,
while leaving all elements outside $J$ as distinct singleton classes.
In this quotient algebra, $J$ serves as the multiplicative zero and the additive maximum element.

The class of ai-semirings includes many well-known and important examples,
such as the max-plus algebra~\cite{aei}, Boolean algebras with join and meet~\cite{bs},
and the semiring of all binary relations on a set~\cite{dol09}.
These and other similar algebras form the algebraic foundation of tropical geometry and optimization~\cite{ms},
underlie classical logic and automata theory~\cite{enderton2001, hopcroft2007},
and play a key role in relational algebra and program verification~\cite{kozen2000, maddux2006}.

A \emph{variety} of ai-semirings is a class of ai-semirings
closed under taking subalgebras, homomorphic images, and arbitrary direct products.
Birkhoff's celebrated theorem tells us that a class of ai-semirings is a variety if and only if it is an equational class;
that is, the class of all ai-semirings satisfying a certain set of identities.
A variety of ai-semirings is \emph{finitely based} if it can be defined by a finite set of identities;
otherwise, it is \emph{nonfinitely based}.
An ai-semiring $S$ is finitely based (resp., nonfinitely based)
if the variety $\mathsf{V}(S)$ it generates is finitely based (resp., nonfinitely based).

The \emph{finite basis problem} for a class of ai-semirings,
one of the most important problems in universal algebra,
concerns the classification of
its members according to whether they are finitely based.
Over the past two decades, this problem has attracted considerable attention and seen substantial progress
(see~\cite{dol07,gr, gpz05,  jac:flat, jrz, pas05, rjzl, rlzc, rlyc, ryy, sr, yr2602, yrg, yrzs}).
For more information on the finite basis problem of other related algebras such as semigroups and monoids,
see the monograph~\cite{lee23}.

Pastijn et al.~\cite{gpz05, pas05} showed that every ai-semiring satisfying the identity $x^2\approx x$ is finitely based.
Dolinka~\cite{dol07} found the first example of a nonfinitely based finite ai-semiring.
Jackson~\cite{jac:flat} solved the finite basis problem for the flat extensions of finite groups,
thereby providing infinitely many nonfinitely based finite ai-semirings.
Shao and Ren~\cite{sr} proved that every ai-semiring in the variety generated by all ai-semirings of order two is finitely based.
Jackson et al.~\cite{jrz} and Zhao et al.~\cite{zrc} classified three-element ai-semirings with respect to the finite basis property.

Moreover, Volkov~\cite{vol21} and Jackson et al.~\cite{jrz} independently
resolved the finite basis problem for the ai-semiring whose multiplicative reduct is the six-element Brandt monoid.
Ren et al.~\cite{rjzl} provided an infinite number of minimal nonfinitely based ai-semiring varieties.
Very recently, the finite basis problem for four-element ai-semirings has been nearly completed, with only two algebras remaining (see \cite{yr2602, yrg, rlyc, rlzc, ryy, yrzs}).

In much of the above work, flat semirings have played an important and even decisive role.
By a \emph{flat semiring} we mean a nontrivial ai-semiring \(S\) with a
multiplicative zero \(0\) that satisfies \(a+b=0\) for all distinct
\(a,b\in S\). Equivalently, an ai-semiring is flat if and only if its
additive order has height one and its top element is the multiplicative
zero \(0\).
As observed in~\cite[Lemma 2.2]{jrz} (see also~\cite[Lemma 4.1.1]{ek}),
a semigroup with zero \(0\) becomes a flat semiring if and only if it is \(0\)-cancellative; that is, for all $a, b, c \in S$,
\[
ab = ac \neq 0 \;\Rightarrow\; b = c \quad\text{and}\quad ab = cb \neq 0 \;\Rightarrow\; a = c.
\]

Let $S$ be a cancellative semigroup, and let $S^0$ be obtained by adjoining a new zero to $S$. Then $S^0$ is $0$-cancellative and hence becomes a flat semiring,
which is called the \emph{flat extension} of $S$ and denoted by $\flat(S)$.

Let $k \geq 1$ be an integer.
A flat semiring is \emph{$k$-nil} if the $k$th power of each element is zero.
A flat semiring is \emph{$k$-nilpotent} if the product of any $k$ elements is zero.
A flat semiring is \emph{nilpotent} if it is $k$-nilpotent for some $k \geq 1$.
It is easy to see that every $k$-nilpotent flat semiring is $k$-nil;
the converse, however, does not hold in general.

It is easy to see that there is a one-to-one, inclusion-preserving
correspondence between semiring congruences and multiplicative ideals on a flat semiring.
This correspondence hinges on the fact that, in a flat semiring,
every multiplicative ideal naturally determines a semiring congruence, since it is an additive filter.
Conversely, every semiring congruence on a flat semiring with zero $0$
has its $0$-class as a multiplicative ideal, and all its other classes
are singletons: if distinct $a,b$ are congruent, then
$a=a+a$ is congruent to $a+b=0$.
This yields the desired one-to-one order-preserving correspondence.

From this observation, we immediately obtain the following subdirect representation.
Let $S$ be a flat semiring with zero $0$, and let $(S_i)_{i\in I}$ be a family of flat semirings.
Then $S$ is isomorphic to a subdirect product of $(S_i)_{i\in I}$
if and only if there exists a family $(J_i)_{i\in I}$ of multiplicative ideals of $S$
such that $S/J_i$ is isomorphic to $S_i$ for each $i \in I$, and the intersection $\bigcap_{i\in I} J_i = \{0\}$.

Next, we introduce an important class of flat semirings. Let \(W\) be a
nonempty set of words in the free semigroup \(X^+\) over \(X\), and let
\(W^{\leq}\) denote the set of all nonempty subwords of words in \(W\).
Let \(S(W)=W^{\leq}\cup\{0\}\), where \(0\) is a new element.
Define a multiplication $\cdot$ on \(S(W)\) by, for \(\bu,\bv\in W^{\leq}\),
\[
\bu\cdot\bv=
\begin{cases}
\bu\bv, & \text{if }\bu\bv\in W^{\leq},\\
0, & \text{otherwise},
\end{cases}
\]
\[
0\cdot 0=0\cdot\bu=\bu\cdot 0=0.
\]
The resulting algebra is a \(0\)-cancellative semigroup, and hence becomes a flat semiring.
If we start the same construction from the free monoid $X^* = X^+\cup \{1\}$ instead,
with $W^{\leq}$ denoting the set of all subwords of words in $W$,
then the corresponding flat semiring is denoted by $M(W)$.
Repeating the construction with the free commutative semigroup $X_c^+$ (or the free commutative monoid $X_c^*$) yields the commutative analogues $S_c(W)$ and $M_c(W)$.

By construction, $S(W)=S(W^{\leq})$. For a singleton set
$W=\{\bw\}$, we write $S(\bw)$. The semiring $S(W)$ is $n$-nilpotent
if and only if every word in $W$ has length less than $n$.
We work throughout in the signature $(+,\cdot)$, without a named zero
or identity element. A subword always means a contiguous subword.

If $W$ consists solely of the empty word, then $M_c(W)$ coincides with $M(W)$,
is a two-element ai-semiring $M_2$, and hence is finitely based (see \cite{sr}).
Jackson et al.~\cite{jrz} proved that, when $W$ is finite and contains at least one nonempty word,
$M_c(W)$ and $M(W)$ are always nonfinitely based, while $S_c(W)$ is
nonfinitely based in most cases. Subsequently, Wu et al.~\cite{wzr}
completely solved the finite basis problem for finite flat semirings of
the form $S_c(W)$. For the noncommutative construction $S(W)$, the results
of~\cite{grz,grzhou2026} leave the case in which $S(W)$ is $4$-nil
and $W$ contains a word of length at least four. We give a criterion
for finite basability that applies to every nonempty word set $W$:
after deleting the first and last letters of each defining word, no
letter may occur twice. This settles the remaining case and yields a
single characterization for the whole class.

To state this characterization precisely, we examine the positions of
repeated letters in the defining words.
A word $\bu$ is \emph{linear} if every variable occurs in $\bu$ at most once;
in particular, the empty word is linear.
For $\bw=x_1\cdots x_d$, where the $x_i$ are not necessarily pairwise distinct, write
\[
 \intr(\bw)=
 \begin{cases}
 x_2\cdots x_{d-1},&d\geq3,\\
 1,&d\leq2.
 \end{cases}
\]
We call this the \emph{strict interior} of $\bw$ and treat the empty word $1$
as linear. Our main theorem is as follows.

\begin{theorem}\label{thm:main}
Let $W$ be any nonempty set of nonempty words. Then the following conditions
are equivalent:
\begin{enumerate}[$(1)$]
\item $S(W)$ is finitely based;
\item $\intr(\bw)$ is linear for every $\bw\in W$;
\item $S(W)$ satisfies $xyz+uyv\approx xyz+uyv+uyz$;
\item $S(W)$ satisfies both $xy^2z\approx0$ and $xyzyt\approx0$.
\end{enumerate}
\end{theorem}

The main idea of the proof is a semantic method for proving that a variety is nonfinitely based.
This method has been widely applied; see, for example, \cite[Theorem 4.9]{jrz}.
For convenience, we summarize it as the following lemma.

\begin{lemma}\label{lem:local-criterion}
Let $\mathcal{V}$ be an arbitrary variety. If for each $n\geq 1$ there exists an algebra $A_n$ such that every
$n$-generated subalgebra of $A_n$ lies in the variety $\mathcal{V}$, but $A_n$ itself does not lie in $\mathcal{V}$,
then $\mathcal{V}$ is nonfinitely based.
\end{lemma}

The paper is organized as follows. Section~\ref{sec:negative} constructs
four classes of finite hypergraph semirings, namely
$S_{\mathbb{H}}$, $S_{\mathbb{H}}^{abc}$, $S_{\mathbb{H}}^{aac}$
and $S_{\mathbb{H}}^{aba}$, which witness nonfinite basability
whenever an internal letter is repeated. The separating identities
hold in every $S(W)$, so the argument applies directly to arbitrary
defining sets. Section~\ref{gen:positive-section} treats words with
linear strict interiors. The main step there uses Dickson's lemma to
bound the number of generators needed to test membership in
$\mathsf{V}(S(W))$.

\section{Preliminaries}\label{sec:preliminaries}
We first recall some terminology concerning identities.
From \cite[Theorem 2.5]{kp} we know that
the set $P_f(X^+)$ of all nonempty finite subsets of $X^+$,
equipped with addition as set-theoretic union and multiplication as elementwise product,
is the free ai-semiring over $X$. Consequently, an \emph{ai-semiring term} (or simply a \emph{term}) is a formal sum of finitely many words in $X^+$.
(In what follows, terms are denoted by bold lowercase letters \(\bu, \bv, \bw, \dots\),
while ordinary lowercase letters \(x, y, z, \dots\) stand for letters.)
The order of the summands in the formal sum is irrelevant,
and repeated occurrences of the same word are identified with a single occurrence.

An \emph{ai-semiring identity} (or simply an \emph{identity}) is
a formal expression of the form \(\bu \approx \bv\), where \(\bu\)
and \(\bv\) are polynomials. We say that an ai-semiring \(S\)
\emph{satisfies} an identity \(\bu \approx \bv\) over
\(X=\{x_1,\ldots,x_n\}\), denoted \(S\vDash \bu\approx \bv\), if
\[
\bu(a_1, a_2, \dots, a_n) = \bv(a_1, a_2, \dots, a_n)
\]
for all \(a_1, a_2, \dots, a_n \in S\). Equivalently,
\(\varphi(\bu) = \varphi(\bv)\) for every semiring homomorphism
\(\varphi \colon P_f(X^+) \to S\). Such a homomorphism is also called
an \emph{assignment} (or a \emph{valuation}), and is uniquely
determined by the images of the elements of \(X\); for convenience,
we may denote it simply by \(\varphi \colon X \to S\).
We also use  \(S \nvDash \bu\approx \bv\) to indicate that \(S\) does not satisfy $\bu\approx \bv$.

Suppose that $\bu = \bu_1 + \bu_2 + \cdots + \bu_m$ and $\bv = \bv_1 + \bv_2 + \cdots + \bv_n$ are two ai-semiring terms.
Then it is easy to verify that $S$ satisfies the identity $\bu \approx \bv$ if and only if
it satisfies the identities $\bu \approx \bu + \bv_i$ and $\bv \approx \bv + \bu_j$ for all $1 \leq i \leq n$, $1 \leq j \leq m$.
Consequently, $S$ lies in an ai-semiring variety $\mathcal{V}$  if and only if
it satisfies every nontrivial identity of $\mathcal{V}$ of the form $\bu \approx \bu + \bq$,
where $\bu$ is a polynomial and $\bq$ is a word. For brevity, we sometimes write \(\bu\approx \bu+\bq\) as \(\bq\preceq \bu\).

For a word \(\bw\) in \(X^+\), let \(c(\bw)\) denote the
\emph{content} of \(\bw\), that is, the set of letters occurring in
\(\bw\), and let \(\ell(\bw)\) be the \emph{length} of \(\bw\), that is,
the number of variables occurring in \(\bw\) counting multiplicities.
For an ai-semiring term $\bu = \bu_1+\cdots+\bu_m$, where $\bu_i\in X^+$,
we let $c(\bu) = \bigcup_{1\leq i\leq m} c(\bu_i)$.

For convenience, we abbreviate the identities
\[
 \bu z\approx z\bu\approx\bu+z\approx\bu,
\]
where $z\notin c(\bu)$, by $\bu\approx0$, and call this the
\emph{$\bu$-free identity}, as in~\cite{rjzl,grz}.

For a word \(\mathbf{w}=w_1w_2\cdots w_n\), its \emph{reversal} is defined by
\[
\mathbf{w}^{\mathrm{rev}} = w_n\cdots w_2w_1.
\]
Similarly, for an ai-semiring term \(\mathbf{u}= \mathbf{u}_1+\cdots+\mathbf{u}_m\), where \(\mathbf{u}_1,\ldots,\mathbf{u}_m\in X^+\), let
\[
\mathbf{u}^{\mathrm{rev}} = \mathbf{u}_1^{\mathrm{rev}}+\cdots+\mathbf{u}_m^{\mathrm{rev}}.
\]
For any semiring \(S\), let \(S^{\mathrm{op}}\) denote its \emph{opposite semiring}: it has
the same underlying set and the same addition as \(S\), but its multiplication
is reversed, that is,
\[
a\cdot_{S^{\mathrm{op}}} b:=b\cdot_S a.
\]
For any ai-semiring identity \(\mathbf{u}\approx \mathbf{v}\), it is easy to verify that
\[
S\vDash \mathbf{u}\approx \mathbf{v} \Leftrightarrow S^{\mathrm{op}}\vDash \mathbf{u}^{\mathrm{rev}}\approx \mathbf{v}^{\mathrm{rev}}.
\]
Thus \(\{\mathbf{u}_{\lambda}\approx \mathbf{v}_{\lambda}\mid \lambda\in \Lambda\}\) is an equational basis of \(S\) if and only if \(\{\mathbf{u}_{\lambda}^{\mathrm{rev}}\approx \mathbf{v}_{\lambda}^{\mathrm{rev}}\mid \lambda\in \Lambda\}\) is an equational basis of \(S^{\mathrm{op}}\).
So the finite basis property is invariant under taking the opposite ai-semiring. Thus we have the following lemma.
\begin{lemma}\label{lem:factor-dual}
The flat semiring \(S(\mathbf{w})\) is finitely based if and only if \(S(\mathbf{w}^{\mathrm{rev}})\) is finitely based.
\end{lemma}

\begin{proof}
It is easy to verify that the opposite semiring of \(S(\mathbf{w})\) is isomorphic to \(S(\mathbf{w}^{\mathrm{rev}})\).
Therefore, by the above analysis, the conclusion follows.
\end{proof}

The following lemma is recalled from~\cite[Proposition 3.12]{grz}.
\begin{lemma}\label{lem:general-subdirect}
The semiring \(S(W)\) is a subdirect product of the semirings
\(S(\mathbf{w})\), \(\mathbf{w}\in W\). Consequently,
\[
 \mathsf{V}(S(W))=\mathsf{V}(S(\mathbf{w})\mid\mathbf{w}\in W).
\]
More generally, whenever \(\mathbf{u}\in W^{\leq}\), the semiring
\(S(\mathbf{u})\) is a homomorphic image of \(S(W)\).
\end{lemma}

The following representation lemma will be used in both directions
of the proof. A nonzero element \(s\) of a semiring \(A\) is said to
\emph{divide} an element \(\omega\in A\) if there exist \(x,y\in A^1\),
where $A^1$ denotes the multiplicative reduct with an identity adjoined if necessary, such that
\[
 xsy=\omega.
\]
We call a nonzero element \(\omega\in A\) a \emph{root} if every nonzero
element of \(A\) divides \(\omega\).

The following lemma provides a sufficient condition for
a flat semiring \(A\) to belong to the variety generated by another flat
semiring \(B\). It will be used repeatedly in the sequel.
\begin{lemma}\label{lem:representation}
Let \(A\) be a flat semiring with a root \(\omega\), let \(B\) be a
flat semiring, and let \(I\) be a nonempty set.
Suppose that tuples \(E_s\in B^I\), indexed by \(s\in A\setminus\{0\}\),
satisfy
\begin{enumerate}[$(\rm i)$]
\item they are pairwise distinct and have no zero coordinate;
\item \(E_sE_t=E_{st}\) whenever \(st\ne0\);
\item for every \(k\geq1\) and nonzero \(s_1,\ldots,s_k\),
\[
 E_{s_1}\cdots E_{s_k}=E_\omega
 \quad\Longrightarrow\quad s_1\cdots s_k=\omega.
\]
\end{enumerate}
Then \(A\in\mathsf{V}(B)\).
\end{lemma}
\begin{proof}
Let \(D\) be the subsemiring of \(B^I\) generated by the displayed tuples
and the zero element \(\underline{0}\). A monomial in these tuples with nonzero
source product \(s\) equals \(E_s\), by (ii).
A monomial with zero source product cannot equal any \(E_s\):
extend \(s\) to \(\omega=xsy\) for some \(x,y\in A^1\), and multiply
by the designated tuples \(E_x,E_y\), omitting the corresponding factor
when a context is the adjoined identity. This would give a monomial equal
to \(E_\omega\) with zero source, contradicting (iii).
Every monomial involving \(\underline{0}\) is certainly \(\underline{0}\).

Set \(J=D\setminus\{E_s:s\ne0\}\).
By distributivity, every element of \(D\) is a finite sum of monomials.
If this sum has no zero coordinate, flat addition forces all its monomial
summands to be the same tuple.
Hence every member of \(J\) either has a zero coordinate or is a monomial
with zero source product.
Zero coordinates persist under addition and under multiplication on
either side. Multiplying a zero-source monomial by another monomial
again has zero source and cannot give a designated tuple.
Distributivity and the observation about sums show that \(J\) is a
multiplicative ideal.

If a sum equals a designated tuple, every summand must equal that tuple
coordinatewise. Consequently \(J\) absorbs addition and is an additive
order filter. It contains the constant zero tuple and is proper.
In \(D/J\), distinct designated tuples add to zero, their nonzero source
products are preserved, and their zero source products are collapsed.
Thus \(D/J\cong A\), proving the claim.
\end{proof}

The equivalence of the last three conditions in Theorem~\ref{thm:main} is easy to verify;
the details are as follows.

\begin{lemma}\label{lem:general-linear-interior-identities}
For a nonempty set \(W\) of nonempty words, the following are equivalent.
\begin{enumerate}[\((1)\)]
\item Every word in \(W\) has linear strict interior.
\item \(S(W)\) satisfies
\begin{equation}\label{gen:M}\tag{M}
 uyz\preceq xyz+uyv.
\end{equation}
\item \(S(W)\) satisfies both \(xy^2z\approx 0\) and \(xyzyt\approx 0\).
\end{enumerate}
In particular, every such $S(W)$ satisfies \(x^4\approx0\).
\end{lemma}

\begin{proof}
We first prove (1)\(\Leftrightarrow\)(2). Every subword of a word with linear strict interior also has linear
strict interior. Suppose (1) holds, and the two words on the right of~\eqref{gen:M} have the same nonzero value \(\mathbf{w}\). The nonempty word
assigned to \(y\) occurs strictly inside \(\mathbf{w}\). Its first letter has a
unique internal occurrence, so the two occurrences of \(y\) start at
the same position. The values of \(x,u\) coincide, as do the values
of \(z,v\), and \(uyz\) also has value \(\mathbf{w}\). If the right-hand side is zero, the inequality is automatic.

Conversely, suppose that \(\mathbf{w}\in W\) has two distinct internal
occurrences of a letter \(a\). Write
\(\mathbf{w}=\mathbf{x} a\mathbf{z}=\mathbf{u} a\mathbf{v}\), where for the second factorization we take the second occurrence,
so that \(\ell(\mathbf{u})>\ell(\mathbf{x})\) and all four contexts are nonempty.
Evaluate \(x,y,z,u,v\) at \(\mathbf{x},a,\mathbf{z},\mathbf{u},\mathbf{v}\), respectively. The two words on the right of~\eqref{gen:M} have value \(\mathbf{w}\). Since the mixed word \(\mathbf{u} a\mathbf{z}\)
has length strictly greater than \(\ell(\mathbf{w})\), its value is either zero
or a different nonzero word in \(S(W)\). But in \(S(W)\), \(0\) is the greatest element,
and distinct nonzero elements are incomparable; hence \eqref{gen:M} fails.
This completes the proof of (1)\(\Leftrightarrow\)(2).

It remains to prove (1)\(\Leftrightarrow\)(3). Any nonzero instance of either displayed
monomial repeats the first letter of the value of \(y\) at two internal
positions. Conversely, consecutive internal occurrences of a
repeated letter give either a subword \(ca^2d\) or a subword \(ca\mathbf{q} ad\)
with \(\mathbf{q}\ne1\), thereby violating the first or the second identity, respectively.

Finally, the substitution \(x=y=z\) in \(xy^2z\approx0\) yields \(x^4\approx0\).
\end{proof}

\section{Nonfinite basability}\label{sec:negative}
We prove that a repeated letter in the strict interior of any word in \(W\)
prevents \(S(W)\) from being finitely based.
The proof uses finite flat semirings whose small generated subsemirings lie in
\(\mathsf{V}(S(W))\), but which fail an identity common to all word semirings.
In the following, we recall the necessary preliminaries on hypergraphs;
for further details, we refer the reader to \cite{jrz}.

A \emph{\(3\)-uniform hypergraph} \(\mathbb{H}\) is a pair
\((V(\mathbb{H}),E(\mathbb{H}))\),
where \(E(\mathbb{H})\) is a family of \(3\)-element subsets of a finite nonempty
set \(V(\mathbb{H})\).
Each element of \(V(\mathbb{H})\) is a \emph{vertex} of \(\mathbb{H}\),
and each element of \(E(\mathbb{H})\) is a \emph{hyperedge} of \(\mathbb{H}\).
Since only \(3\)-uniform hypergraphs are involved in this paper,
we shall sometimes simply call them hypergraphs.

For a nonempty set $F\subseteq E(\mathbb H)$, the
\emph{edge-induced subhypergraph} $\mathbb H\langle F\rangle$ determined
by $F$ is
\[
\textstyle \mathbb H\langle F\rangle
 =\left(\bigcup_{e\in F}e,\,F\right),
\]
that is, its vertex set is
$V(\mathbb H\langle F\rangle) = \bigcup_{e\in F}e$ and its hyperedge set is
$E(\mathbb H\langle F\rangle) = F$.
It is \emph{proper} if $F\subsetneq E(\mathbb H)$.
Throughout this section, subhypergraphs are always understood in this
edge-induced sense.

A \emph{path} of length \(n\geq 1\) in a hypergraph \(\mathbb{H}\) is an
alternating sequence
\begin{equation}\label{path}
v_1, e_1, v_2, e_2, \ldots, v_n, e_n, v_{n+1}\tag{$*$}
\end{equation}
of distinct vertices and hyperedges such that \(v_i,v_{i+1}\in e_i\) for
each \(1\leq i\leq n\). Analogously to the corresponding notion in graph
theory, one can partition \(\mathbb{H}\) into connected components
according to whether two vertices are joined by a path. If \(\mathbb{H}\)
has only one connected component, that is, if any two distinct vertices
of \(\mathbb{H}\) are joined by a path, then \(\mathbb{H}\) is said to be
\emph{connected}.

If \(n\geq 2\) and we set \(v_{n+1}=v_1\) in the sequence~\eqref{path} while
keeping the other conditions unchanged (so that \(v_1\in e_n\cap e_1\)),
then the sequence is called a \emph{cycle} of length \(n\).

The \emph{girth} of \(\mathbb{H}\), denoted by \(g(\mathbb{H})\), is the
length of its shortest cycles. If \(\mathbb{H}\) contains no cycles, then
\(\mathbb{H}\) is called a \emph{hyperforest}, and we set
\(g(\mathbb{H})=\infty\).

From now on we shall always consider hypergraphs with
\(g(\mathbb{H})\geq 4\). This condition ensures that any two distinct
hyperedges of \(\mathbb{H}\) intersect in at most one vertex (such a
hypergraph is \emph{linear}), and that a set \(\{v_1,v_2,v_3\}\) of
vertices is a subhyperedge (that is, a subset of some hyperedge) if and
only if every \(2\)-element subset of \(\{v_1,v_2,v_3\}\) is a
subhyperedge (see~\cite[Lemma 3.2]{jrz}).

Let \(k\geq 2\). A hypergraph \(\mathbb{H}\) is \emph{\(k\)-colourable}
if there exists a mapping \(\varphi\) from \(V(\mathbb{H})\) to a
\(k\)-element colouring set such that \(|\varphi(e)| \geq 2\) for all
\(e \in E(\mathbb{H})\), that is, no hyperedge is monochromatic.
The smallest such \(k\) for which \(\mathbb{H}\) is \(k\)-colourable is
called the \emph{chromatic number} of \(\mathbb{H}\), denoted by
\(\chi(\mathbb{H})\).

\begin{remark}\label{hypergraphs}
We remind the reader once again that, throughout this section, all
hypergraphs are $3$-uniform, have a nonempty finite vertex set, have
no isolated vertices, and have girth at least four. Every subhypergraph
is induced by a nonempty subset of the edge set and therefore satisfies
the same conditions.
\end{remark}

A hypergraph is \emph{just non-$2$-colourable} if it is not
$2$-colourable but every proper nonempty edge-induced subhypergraph is
$2$-colourable. Thus minimality is taken with respect to deletion of
edges, with vertices incident to no remaining edge discarded.

\begin{lemma}\label{lem:colourings}
Let \(\mathbb{F}\) be a hyperforest, and put
\[
 \Colorings(\mathbb{F})
 =\left\{\varepsilon:V(\mathbb{F})\to\{0,1\}\mid
                  (\forall\,e\in E(\mathbb{F}))~
                  \sum_{v\in e}\varepsilon(v)=1\right\}.
\]
Then the following statements hold.
\begin{enumerate}[$(1)$]
\item Any prescribed values on at most two vertices extend to a member
of \(\Colorings(\mathbb{F})\), unless two vertices in a common edge are
both prescribed \(1\).
In particular, each vertex takes both values \(0\) and \(1\), and distinct
vertices are separated by \(\Colorings(\mathbb{F})\).
\item For any \(u,v,w\in V(\mathbb{F})\), repetitions allowed,
\[
 \varepsilon(u)+\varepsilon(v)+\varepsilon(w)=1
 \quad\text{for all }\varepsilon\in\Colorings(\mathbb{F})
\]
if and only if the three vertices are distinct and form an edge.
\end{enumerate}
\end{lemma}

\begin{proof}
For (1), interchanging zero and one in a member of
$\Colorings(\mathbb F)$ gives a $2$-in-$3$ satisfaction
in the convention of~\cite{jrz}.
By Lemma~3.3 of that paper, every $3$-uniform hyperforest
is $\leq2$-robustly $2$-in-$3$ satisfiable.
Interchanging the two colours yields assertion (1).

For (2), an actual edge satisfies the equality by definition.
A repeated vertex can be assigned value one using (1), making the sum
at least two.
For three distinct vertices not forming an edge, some pair lies in no
common edge. Otherwise the three pair-containing edges would be distinct
and form a cycle of length three (if two coincided, they would contain
the whole triple). Assign value one to that pair using (1); the sum is
then at least two.
\end{proof}

\begin{lemma}\label{lem:marginals}
For every  hyperforest \(\mathbb{F}\), there is a
probability distribution on \(\Colorings(\mathbb{F})\) such that
\[
 \mathbb{E}[\varepsilon(v)]=\frac13
 \qquad(v\in V(\mathbb{F})).
\]
\end{lemma}
\begin{proof}
Root each component of the incidence forest at a vertex.
Give the root value one with probability \(1/3\).
Process edges away from the root. Each edge has one previously assigned
parent vertex and two unassigned child vertices.
If the parent has value one, give both children value zero.
If the parent has value zero, choose \((1,0)\) or \((0,1)\) for the
children with equal probabilities. This always produces an exact-one
colouring.
Inductively, each child has probability
\((1-1/3)/2=1/3\) of receiving value one.
The same construction is carried out independently on each component.
\end{proof}

\begin{lemma}\label{lem:kernel}
Let \(\mathbb{H}\) be a just non-$2$-colourable hypergraph.
\begin{enumerate}[$(1)$]
\item \(\mathbb{H}\) is connected.
\item If a family \((q_v)_{v\in V(\mathbb{H})}\) of real numbers satisfies
\[
 \sum_{v\in e}q_v=0
\]
for each \(e\in E(\mathbb{H})\), then \(q_v=0\) for every
\(v\in V(\mathbb{H})\).
\end{enumerate}
\end{lemma}

\begin{proof}
For (1), if $\mathbb H$ were disconnected, at least one component
would be non-$2$-colourable. Since there are no isolated vertices,
that component would be a proper nonempty edge-induced subhypergraph,
contrary to edge minimality.

For (2), suppose that the weights are not all zero, and put
\[
 Z=\{v\in V(\mathbb H)\mid q_v=0\},\qquad
 F_Z=\{e\in E(\mathbb H) \mid e\subseteq Z\}.
\]
The set $F_Z$ is a proper subset of $E(\mathbb H)$: a vertex of
nonzero weight lies in an edge, and that edge is not contained in $Z$.
If $F_Z\ne\varnothing$, edge minimality gives a proper $2$-colouring
of $\mathbb H\langle F_Z\rangle$. Extend it arbitrarily to the
vertices in $Z\setminus\bigcup_{e\in F_Z}e$. If $F_Z=\varnothing$,
assign colours to the vertices of $Z$ arbitrarily. In either case,
we obtain a map $\psi:Z\to\{+,-\}$ that is nonconstant on every
edge contained in $Z$; this includes $Z=\varnothing$.

Colour each vertex in $Z$ according to $\psi$, each vertex of positive
weight by $+$, and each vertex of negative weight by $-$. Every edge
contained in $Z$ is non-monochromatic by construction. Any other edge
has at least one nonzero weight; since its weights sum to zero, it
contains both a positive and a negative weight. Hence it too is
non-monochromatic. This gives a proper $2$-colouring of $\mathbb H$,
a contradiction. Thus all weights are zero.
\end{proof}

\begin{lemma}\label{lem:existence}
For every \(m\geq 3\), there exists a just non-$2$-colourable
hypergraph of girth greater than \(m\).
\end{lemma}
\begin{proof}
By the Erd\H{o}s--Hajnal theorem~\cite[Corollary~13.4]{eh}, there is
a finite $3$-uniform hypergraph $\mathbb G$ of girth greater than $m$
that is not $2$-colourable. Discarding any isolated vertices does not
affect either property. Choose a nonempty edge set
$F\subseteq E(\mathbb G)$ of least cardinality such that
$\mathbb G\langle F\rangle$ is not $2$-colourable.
Then $\mathbb H=\mathbb G\langle F\rangle$ has no isolated vertices
and is just non-$2$-colourable. Every cycle in $\mathbb H$
is a cycle in $\mathbb G$, so $g(\mathbb H)\geq g(\mathbb G)>m$.
\end{proof}

Let $\mathbb H=(V,E)$ satisfy the standing conventions in Remark~\ref{hypergraphs}.
We begin with the classical construction $S_{\mathbb H}$.
Recall from~\cite{jrz} that the \emph{hypergraph semiring}
\(S_{\mathbb{H}}\) defined by \(\mathbb{H}\) is the flat semiring
with multiplicatively absorbing element \(0\), generated by a copy
\(\{a_v \mid v\in V\}\) of \(V\), and presented by the
following relations (R1)--(R4):
\begin{itemize}
\item[(R1)] \(a_ua_v=a_va_u\)
for all \(u,v\in V\);

\item[(R2)] \(a_ua_v=0\) if \(\{u,v\}\) is not a
\(2\)-element subset of a hyperedge in \(E\);

\item[(R3)] \(a_{u}a_{v}a_{w}
=a_{u'}a_{v'}a_{w'}\)
if \(\{u,v,w\},\{u',v',w'\}\in E\);

\item[(R4)] \(a_{u}a_{v}
=a_{u'}a_{v'}\)
if \(\{u,v,w\},\{u',v',w\}\in E\) for some \(w\in V\).
\end{itemize}
We denote the common nonzero value in (R3) by $\omega$, and let $\ell_v$ denote
the common product of the other two generators in an edge containing
\(v\). In each of the four constructions, $\omega$ is the common
value of the edge products and is a root.

It is easy to verify that
\[
 S_{\mathbb{H}}=\{0,\omega\}\dotcup\{a_v,\ell_v\mid v\in V\}
\]
is a \emph{commutative} flat semiring, and its nonzero binary products are precisely
\begin{gather}\label{eq:K-products}
a_ua_v=\ell_w\quad(\{u,v,w\}\in E),\qquad
 a_v\ell_v=\ell_va_v=\omega\quad(v\in V).
\end{gather}
This construction has been applied successfully to solve the finite
basis problem for a class of semirings; see~\cite{jrz,gjrz,grsy,grzhou2026}.
However, most of the semirings $S(W)$ considered in this paper are
noncommutative. Motivated by the above idea, we need to introduce
three further classes of \emph{noncommutative} hypergraph semirings
\(S^{abc}_{\mathbb{H}}\), \(S^{aac}_{\mathbb{H}}\), and
\(S^{aba}_{\mathbb{H}}\), as follows.
\begin{itemize}
\item \(S_{\mathbb{H}}^{abc}=\{0,{\omega}\}\dotcup
\{a_v,b_v,c_v,\ell_v,r_v\mid v\in V\}\)
with nonzero products precisely

\begin{align}
a_ub_v=\ell_w,\quad b_vc_w=r_u
 &\quad(\{u,v,w\}\in E),\nonumber\\
 a_vr_v=\ell_vc_v={\omega}&\quad(v\in V).
\end{align}

\item \(S_{\mathbb{H}}^{aac}=\{0,{\omega}\}\dotcup
\{a_v,c_v,\ell_v,r_v\mid v\in V\}\) with nonzero products precisely
\begin{align}
a_ua_v=\ell_w,\quad a_vc_w=r_u
 &\quad(\{u,v,w\}\in E),\nonumber\\
 a_vr_v=\ell_vc_v={\omega}&\quad(v\in V).
\end{align}

\item \(S_{\mathbb{H}}^{aba}=\{0,{\omega}\}\dotcup
\{a_v,b_v,\ell_v,r_v\mid v\in V\}\),
with nonzero products precisely
\begin{align}
a_ub_v=\ell_w,\quad b_va_w=r_u
 &\quad(\{u,v,w\}\in E),\nonumber\\
 a_vr_v=\ell_va_v={\omega}&\quad(v\in V).
\end{align}
\end{itemize}
{Each edge rule applies to all six orderings of the edge, and every
unlisted binary product is zero. The matching products apply at every
vertex. These rules define associative, $0$-cancellative multiplications.
Indeed, a nonzero product of three vertex labels occurs precisely in the
displayed type order on an edge, and both bracketings give $\omega$;
every other triple product is zero. For a fixed left or right factor,
{linearity ensures that no two distinct choices of the other factor give the same nonzero product.}
Thus flat addition makes each construction a flat semiring. Every
nonzero element divides ${\omega}$.}

{Fix a finite \(3\)-uniform hyperforest \(\mathbb{F}\) without isolated vertices.}
{The coordinate powers below} are indexed by
\(\Colorings(\mathbb{F})\).
We shall apply Lemma~\ref{lem:representation} to each construction.

{We begin with the representation of the classical hypergraph semiring $S_{\mathbb F}$.}

In \(S(a^4)^{\Colorings(\mathbb{F})}\), set
\begin{equation}\label{eq:K-coord}
 A_v(\eps)=a^{1+\eps(v)},\qquad
 {L_v(\eps)=a^{3-\eps(v)},\qquad}
 W(\eps)=a^4.
\end{equation}
{These correspond to \(a_v,\ell_v,\omega\) in \(S_{\mathbb{F}}\).}

\begin{proposition}\label{prop:K-forest}
One has \(S_{\mathbb{F}}\in\mathsf{V}(S(a^4))\).
\end{proposition}
\begin{proof}

Each coordinate is nonzero.
Within the \(A\)- and \(L\)-families use vertex separation.
Choosing \(\eps(v)=0\) gives \(A_v(\eps)=a\), different from every
\(L\)-value.
The root has length four and is distinct from both families.
The edge condition gives \(A_uA_v=L_w\), and \(A_vL_v=W\).
Thus all nonzero products are preserved.

The target is commutative, so for root reflection we may order types
without loss of generality. Write \(e_u=\eps(u)\).
For products with at least two factors, the complete list allowed by
minimum lengths is
\[
\begin{array}{c@{\qquad}l}
\hline
\text{Product}&\text{Condition for equality with \(W\)}\\
\hline
A_uA_v&e_u+e_v=2\\
A_uL_v&e_u=e_v\\
L_uL_v&e_u+e_v=2\\
A_uA_vA_w&e_u+e_v+e_w=1\\
A_uA_vL_w&e_w=1+e_u+e_v\\
A_uA_vA_wA_\ell&e_u+e_v+e_w+e_\ell=0\\
\hline
\end{array}
\]
A factor \(W\) is impossible in such a product.
The first, third, fifth, and last rows fail for some colouring by
Lemma~\ref{lem:colourings}(1).
The second forces \(u=v\), and the fourth forces an edge.
These are precisely the source products with value \(\omega\).
Apply Lemma~\ref{lem:representation}.

\end{proof}

{We next treat $S_{\mathbb F}^{abc}$, then $S_{\mathbb F}^{aac}$ and $S_{\mathbb F}^{aba}$.}

{Assume \(c,d\ne a\), allowing \(c=d\).}
In \(S(ca^2d)^{\Colorings(\mathbb{F})}\), define
\begin{align}
 A_v(\eps)&=ca^{\eps(v)},&
 B_v(\eps)&=a^{1+\eps(v)},&
 C_v(\eps)&=a^{\eps(v)}d,\label{eq:T-coord}\\
 L_v(\eps)&=ca^{2-\eps(v)},&
 R_v(\eps)&=a^{2-\eps(v)}d,&
 W(\eps)&=ca^2d.\nonumber
\end{align}
They correspond to the labels \(a_v,b_v,c_v,\ell_v,r_v,{\omega}\) of
\(S^{abc}_{\mathbb{F}}\).
Expressions \(ca^0\) and \(a^0d\) mean \(c\) and \(d\), respectively;
they do not introduce an empty-word element.

\begin{proposition}\label{prop:T-forest}
Let \(a,c,d\) be letters with \(c,d\neq a\), where \(c\) and \(d\)
may be equal. Then
\[
S^{abc}_{\mathbb{F}}\in\mathsf{V}(S(ca^2d)).
\]
\end{proposition}
\begin{proof}

All coordinates in \eqref{eq:T-coord} are nonzero.
Within each indexed family, separation follows from
Lemma~\ref{lem:colourings}(1).
The \(B\)-values contain no letter different from \(a\);
the \(A,L,R,C\)-values contain exactly one such occurrence, and
\(W\) contains two, even if \(c=d\).
Choosing \(\eps(v)=0\) separates \(A_v\) from every \(L\)-tuple,
and \(C_v\) from every \(R\)-tuple, by length.
Choosing \(\eps(v)=1\) gives \(A_v(\eps)=ca\), separating it from every
\(C\)-value, which ends in \(d\ne a\), and every \(R\)-value,
which starts in \(a\ne c\).
Likewise \(C_v(\eps)=ad\) separates it from every \(L\)-value, which
starts in \(c\ne a\). Finally, \(L\)-values start in \(c\), whereas
\(R\)-values start in \(a\). Thus all designated tuples are distinct.

The edge equality \(\eps(u)+\eps(v)+\eps(w)=1\) gives
\[
 A_uB_v=L_w,\qquad B_vC_w=R_u,
\]
and the matching products give \(A_vR_v=L_vC_v=W\).
All nonzero products are preserved.

For root reflection, suppose a product of at least two designated
tuples equals \(W\) in every coordinate. No factor is \(W\), by
length.
Since the root contains exactly two occurrences of letters other than
\(a\), exactly two factors have a type other than \(B\).
These must be the first and last factors.
The first cannot be \(R\), which starts in \(a\), or \(C_v\),
which starts in \(a\) when \(\eps(v)=1\).
Thus its type is \(A\) or \(L\).
Dually, the last type is \(C\) or \(R\).
This argument also applies when \(c=d\).

Write \(e_u=\eps(u)\).
Minimum numbers of occurrences of \(a\) leave precisely the following
possibilities; the right-hand conditions must hold for every colouring:
\[
\begin{array}{c@{\qquad}l}
\hline
\text{Product}&\text{Condition for equality with \(W\)}\\
\hline
A_uC_v&e_u+e_v=2\\
A_uR_v&e_u=e_v\\
L_uC_v&e_u=e_v\\
L_uR_v&e_u+e_v=2\\
A_uB_vC_w&e_u+e_v+e_w=1\\
A_uB_vR_w&e_w=1+e_u+e_v\\
L_uB_vC_w&e_u=1+e_v+e_w\\
A_uB_vB_wC_\ell&e_u+e_v+e_w+e_\ell=0\\
\hline
\end{array}
\]
Indeed, there are at most two intermediate \(B\)-factors.
The omitted shape \(LBR\), and any two-\(B\) shape with a \(L\)- or
\(R\)-endpoint, contain at least three copies of \(a\).

The first and fourth rows are impossible by assigning zero to one
displayed vertex. The sixth and seventh are impossible by assigning
zero to the vertex on the left. The last is impossible by assigning
one to a displayed vertex. By Lemma~\ref{lem:colourings}, rows two and
three force \(u=v\), while row five forces an edge.
The remaining products therefore have source value \(\omega\).
For a one-factor product, root reflection follows from separation.
Lemma~\ref{lem:representation} now gives the required membership.

\end{proof}

\begin{proposition}\label{prop:internal-repeat-coordinates}
Let \(a,c,d\) be letters and let \(\bq\) be a nonempty word containing
no \(a\).
Suppose that it is not the case that \(\bq=c=d\) is a single letter.
Put \(\bw=ca\bq ad\). For every finite \(3\)-uniform hyperforest
\(\mathbb{F}\) without isolated vertices,
\[
 S^{abc}_{\mathbb{F}}\in\mathsf{V}(S(\bw)).
\]
\end{proposition}
\begin{proof}

Index coordinates by the exact-one colourings of \(\mathbb{F}\).
For a vertex \(u\), writing \(e_u=\eps(u)\), define the tuples by the
following table:
\[
\begin{array}{c@{\qquad}cc}
\hline
 &e_u=0&e_u=1\\\hline
 A_u&c&ca\bq\\
 B_u&a&a\bq a\\
 C_u&d&\bq ad\\
 L_u&ca\bq a&ca\\
 R_u&a\bq ad&ad\\
\hline
\end{array}
\qquad W(\eps)=ca\bq ad.
\]

Every entry is a nonempty subword of \(\bw\).
An immediate concatenation check gives
\[
 A_uB_v=L_w,\quad B_vC_w=R_u
 \quad(\{u,v,w\}\in E(\mathbb{F})),\qquad
 A_uR_u=L_uC_u=W.
\]
Thus all nonzero source products are preserved.

Put \(h=\ell(\bq)+1\geq2\).
The basic tuples \(A_u,B_u,C_u\) have length \(1+he_u\),
whereas the derived tuples \(L_u,R_u\) have length \(2+h(1-e_u)\).
Two equal tuples in the same class have the same index, by separation
of vertices by colourings. A basic and a derived tuple cannot have
equal lengths: equality would require \(h(e_u+e_v)=h+1\).
No nonroot tuple has the root length \(h+3\).

For equal indices the three basic types are distinct.
Indeed, equality of \(A_u,B_u\) would give \(c=a\) and
\(aa\bq=a\bq a\), contradicting the absence of \(a\) from \(\bq\).
Equality of \(B_u,C_u\) similarly contradicts the first letter of
\(\bq\).
Equality of \(A_u,C_u\) would give \(c=d\) and \(ca\bq=\bq ac\).
If \(\ell(\bq)\geq2\), comparison of the second letters is impossible;
if \(\ell(\bq)=1\), it gives precisely the excluded case
\(\bq=c=d\).
Finally, equality of \(L_u,R_u\) would give \(ca=ad\), hence
\(c=d=a\), and then \(aa\bq a=a\bq aa\), again impossible.
All designated tuples are therefore pairwise distinct.

We record the elementary word comparisons that will identify root
products. Among the three long basic words
\(ca\bq,a\bq a,\bq ad\), the only one equal to the prefix \(ca\bq\)
of \(\bw\) is the first, the only one equal to the middle subword
\(a\bq a\) is the second, and the only one equal to the suffix
\(\bq ad\) is the third.
For comparisons involving \(a\bq a\), this follows by comparing its
initial two letters with those of the other candidates, using that
\(\bq\) contains no \(a\).
The remaining possible equality \(ca\bq=\bq ad\) is impossible by
comparison of the second letters when \(\ell(\bq)\geq2\), and for
\(\ell(\bq)=1\) gives precisely the excluded case.
Also the long derived words \(ca\bq a\) and \(a\bq ad\) are distinct:
equality forces \(c=a\) and then forces the first letter of \(\bq\)
to be \(a\).

Suppose a product of designated nonzero tuples equals \(W\).
For a one-factor product, separation already gives the assertion of
root reflection. For at least two factors, none can be \(W\).
Let \(n\) count its basic factors and \(m\) its derived factors.
Use the distribution from Lemma~\ref{lem:marginals}, with every
vertex having marginal probability \(1/3\) of colour \(1\).
Its expected length is
\[
 n\left(1+\frac h3\right)
 +m\left(2+\frac{2h}3\right)
 =(n+2m)\frac{h+3}{3}.
\]
Equality with the constant root length \(h+3\) gives \(n+2m=3\).

For a product of three basic factors, equality of lengths says that
their three index colours sum to one in every coordinate. The
indices are therefore distinct and form an edge.
Each of these three indices can be assigned colour one, forcing the
other two colours to be zero.
The three word comparisons above, applied successively at those
coordinates, force the types to be \(A,B,C\) in that order.
The source product is therefore \(\omega\).

For a product of one basic and one derived factor, equality of lengths
forces their indices to agree. When the basic factor is first, its
long value must be the long prefix \(ca\bq\), so it has type \(A\);
the derived factor's long value must be the suffix \(a\bq ad\), so it
has type \(R\). When the derived factor is first, the corresponding
comparisons force types \(L,C\).
Thus the product is \(A_uR_u\) or \(L_uC_u\), and again the source
product is \(\omega\).
This proves root reflection, so Lemma~\ref{lem:representation} gives
\(S^{abc}_{\mathbb{F}}\in\mathsf{V}(S(\bw))\).

\end{proof}

Let $a\ne b$. In $S(a^3b)^{\Colorings(\mathbb F)}$, define
\begin{align}
 A_v(\eps)&=a^{1+\eps(v)},&
 C_v(\eps)&=a^{\eps(v)}b,\label{eq:R-coord}\\
 L_v(\eps)&=a^{3-\eps(v)},&
 R_v(\eps)&=a^{2-\eps(v)}b,&
 W(\eps)&=a^3b.\nonumber
\end{align}
These tuples correspond to \(a_v,c_v,\ell_v,r_v,\omega\) in
\(S^{aac}_{\mathbb{F}}\).

\begin{proposition}\label{prop:R-forest}
For \(a\ne b\), one has \(S^{aac}_{\mathbb{F}}\in\mathsf{V}(S(a^3b))\).
\end{proposition}
\begin{proof}

All coordinates are nonzero.
Within each indexed family use Lemma~\ref{lem:colourings}(1).
A colouring with \(\eps(v)=0\) gives \(A_v(\eps)=a\), separating
\(A_v\) from all \(L\)-tuples; it gives \(C_v(\eps)=b\), separating
\(C_v\) from all \(R\)-tuples. The presence of \(b\), and the length of
\(a^3b\), separate the remaining families.
The edge equalities give \(A_uA_v=L_w\) and \(A_uC_v=R_w\).
Also \(A_vR_v=L_vC_v=W\).

For a product of at least two designated tuples to equal \(W\),
exactly one factor has type \(C\) or \(R\), and it is last.
All earlier factors have type \(A\) or \(L\); a factor \(W\) is
excluded by length.
With \(e_u=\eps(u)\), the complete list permitted by minimum lengths is
\[
\begin{array}{c@{\qquad}l}
\hline
\text{Product}&\text{Condition for equality with \(W\)}\\
\hline
A_uC_v&e_u+e_v=2\\
A_uR_v&e_u=e_v\\
L_uC_v&e_u=e_v\\
L_uR_v&e_u+e_v=2\\
A_uA_vC_w&e_u+e_v+e_w=1\\
A_uA_vR_w&e_w=1+e_u+e_v\\
A_uL_vC_w&e_v=1+e_u+e_w\\
L_uA_vC_w&e_u=1+e_v+e_w\\
A_uA_vA_wC_\ell&e_u+e_v+e_w+e_\ell=0\\
\hline
\end{array}
\]
Only rows two, three, and five can hold for every colouring.
The others are excluded by assigning zero to a vertex required to be
constantly one, or one to a vertex required to be constantly zero.
In rows six through eight, assign zero to the vertex on the left.
The surviving rows force equal indices in the matching products,
or an edge in the three-factor product.
Their source value is \(\omega\).
One-factor reflection follows from separation, so
Lemma~\ref{lem:representation} applies.

\end{proof}

\begin{proposition}\label{prop:J-forest}
For distinct letters \(a,b\), one has
\(S^{aba}_{\mathbb{F}}\in\mathsf{V}(S(ababa))\).
\end{proposition}
\begin{proof}

For the hyperforest \(\mathbb{F}\), define tuples over \(S(ababa)\)
by the following table, where \(e_u=\eps(u)\):
\[
\begin{array}{c@{\qquad}cc}
\hline
 &e_u=0&e_u=1\\
\hline
A_u&a&aba\\
B_u&b&bab\\
L_u&abab&ab\\
R_u&baba&ba\\
\hline
\end{array}
\qquad W(\eps)=ababa.
\]
The tuples correspond to $a_v,b_v,\ell_v,r_v,\omega$, respectively.
Every coordinate is nonzero. For an edge $\{u,v,w\}$, direct
concatenation gives $A_uB_v=L_w$ and $B_vA_w=R_u$.
Moreover, $A_vR_v=L_vA_v=W$. Hence all nonzero source products are
preserved.

The expected lengths \(5/3\), \(10/3\), and \(5\) separate the short
types \(A,B\), the long types \(L,R\), and the root. Within each
group, length comparison first determines the index; the initial
letter then separates the two types. Thus the tuples are pairwise
distinct.

Suppose a product of at least two designated tuples equals \(W\).
No factor is \(W\). By Lemma~\ref{lem:marginals}, a short tuple has
expected length \(5/3\) and a long tuple has expected length \(10/3\).
If there are \(r\) short factors and \(s\) long factors, then
\(r+2s=3\).
Thus there are three short factors, or one short and one long factor.
In the three-factor case, the length equation forces the indices to
form an edge. The first and last types must be \(A\), because every
\(B\)-value begins and ends in \(b\).
If the middle type were also \(A\), assigning colour one to its index
would give \(aabaa\), not \(ababa\).
Thus the product is \(A_uB_vA_w\) on an edge.
In the two-factor case, length equality forces the same index.
For a short--long product, a leading \(B\) gives the wrong initial
letter, and \(AL\) gives \(aabab\) at colour zero; only \(AR\)
remains. For a long--short product, a leading \(R\) gives the wrong
initial letter, and \(LB\) ends in \(b\); only \(LA\) remains.
These are exactly the source products with value \(\omega\).
Lemma~\ref{lem:representation} proves
\(S^{aba}_{\mathbb{F}}\in\mathsf{V}(S(ababa))\).

\end{proof}

\begin{lemma}\label{lem:local-membership}
Let \(m\geq1\), and put
\[
 g_m={\max\left\{3,\binom m2+m\right\}}.
\]
Let \(\mathbb{H}\) be a \(3\)-uniform hypergraph without isolated
vertices and of girth greater than \(g_m\). For each of the four
constructions, every at most \(m\)-generated subsemiring not contained
in \(\{0,\omega\}\) is isomorphic to a subsemiring of the corresponding
algebra on a finite hyperforest without isolated vertices.

In particular:
\begin{enumerate}[$(1)$]
\item every at most $m$-generated subsemiring of $S_{\mathbb H}$ lies
in $\mathsf V(S(a^4))$;
\item every at most $m$-generated subsemiring of $S_{\mathbb H}^{abc}$
lies in $\mathsf V(S(ca^2d))$, for $c,d\ne a$, allowing $c=d$;
\item every at most $m$-generated subsemiring of $S_{\mathbb H}^{aac}$
lies in $\mathsf V(S(a^3b))$, for $a\ne b$;
\item every at most $m$-generated subsemiring of $S_{\mathbb H}^{aba}$
lies in $\mathsf V(S(ababa))$, for $a\ne b$.
\end{enumerate}

\end{lemma}
\begin{proof}
{The girth assumption implies girth at least four, and hence linearity.} We give the argument for all
four constructions at once.
Let \(G\) be a generating set of size at most \(m\), and let \(U\) be
the set of vertex indices of its labels, disregarding \(0,\omega\).
Then \(|U|\leq m\).
If \(U=\varnothing\), the generated subsemiring is contained in
\(\{0,\omega\}\), which is isomorphic to the subsemiring consisting
of zero and the full root word in each corresponding flat semiring of
the form \(S(W)\).

Otherwise, select every edge containing at least two vertices of
\(U\), and one incident edge for each vertex of \(U\).
By linearity, at most
\(\binom{|U|}{2}+|U|\leq g_m\) edges are selected.
{Let $F$ be the set of selected edges and put
$\mathbb F=\mathbb H\langle F\rangle$. This is a nonempty
edge-induced subhypergraph and has no isolated vertices. It is a
hyperforest, since any cycle in it would have length at most
$|F|\leq g_m<g(\mathbb H)$.}

{Multiplication creates no $a$-label in $S_{\mathbb H}$;
no $a$-, $b$-, or $c$-label in $S_{\mathbb H}^{abc}$;
no $a$- or $c$-label in $S_{\mathbb H}^{aac}$; and no $a$- or
$b$-label in $S_{\mathbb H}^{aba}$.}
Therefore every newly produced \(\ell\)- or \(r\)-label arises from an
edge rule using two original generator indices in \(U\); that edge
was selected.
All remaining nonzero products depend only on matching indices and
yield \(\omega\). Flat addition yields an existing label or zero.

Consequently, interpreting the same generators in
{\(S_{\mathbb{F}},S_{\mathbb{F}}^{abc},S_{\mathbb{F}}^{aac},
S_{\mathbb{F}}^{aba}\)}, respectively, gives exactly the same term
values and equalities as in
{\(S_{\mathbb{H}},S_{\mathbb{H}}^{abc},S_{\mathbb{H}}^{aac},
S_{\mathbb{H}}^{aba}\)}. The corresponding generated subsemirings are
isomorphic.
{Apply Propositions~\ref{prop:K-forest}, \ref{prop:T-forest},
\ref{prop:R-forest}, and~\ref{prop:J-forest}, and closure under}
subsemirings.
\end{proof}

We next construct an identity valid in every \(S(W)\) that excludes
the four test semirings.

Let \(\mathbb{H}=(V,E)\) be a {just non-$2$-colourable}
hypergraph. Choose distinct vertices \(s,t\in V\) and four disjoint
families of variables \(x_v,y_v,z_v,h_v\), indexed by \(v\in V\).
Put
\begin{equation}\label{eq:U}
 \bu_{\mathbb{H}}=
 \sum_{\{u,v,w\}\in E}
       x_u y_v z_w
 \;+\;\sum_{v\in V}x_vh_v.
\end{equation}
The first sum contains all six orderings of each edge.
Consider the identity
\begin{equation}\label{eq:sigma}
 \sigma_{\mathbb{H}}\colon
 \bu_{\mathbb{H}}\approx \bu_{\mathbb{H}}+x_sh_t.
\end{equation}

\begin{proposition}\label{prop:universal}
Every flat semiring \(S(W)\), for an arbitrary nonempty set
\(W\subseteq X^+\), satisfies \(\sigma_{\mathbb{H}}\).
{Under the standing girth assumption, each of}
\(S_{\mathbb{H}}, S_{\mathbb{H}}^{abc},S_{\mathbb{H}}^{aac},
S_{\mathbb{H}}^{aba}\) fails \(\sigma_{\mathbb{H}}\).
\end{proposition}
\begin{proof}
Take an evaluation in \(S(W)\). If \(\bu_{\mathbb{H}}\) has value
zero, the identity holds because zero absorbs addition.
Otherwise, flatness makes every summand of \eqref{eq:U} equal to one
nonzero word \(\br\).
All variables have nonzero values, since every vertex lies in an
edge.
Write their word values as \(\bx_v,\by_v,\bz_v,\bh_v\), and put
\[
 \alpha_v=\ell(\bx_v),\qquad \beta_v=\ell(\by_v),\qquad
 \gamma_v=\ell(\bz_v),\qquad L=\ell(\br).
\]
For every ordered edge,
\begin{equation}\label{eq:lengths}
 \alpha_u+\beta_v+\gamma_w=L.
\end{equation}
Swapping the first two entries shows that \(\alpha_v-\beta_v\) is
constant on each edge; swapping the last two shows the same for
\(\beta_v-\gamma_v\).
By connectedness in Lemma~\ref{lem:kernel}, there are global
constants \(A,B\) such that
\[
 \alpha_v=\beta_v+A,\qquad \gamma_v=\beta_v-B
 \quad(v\in V).
\]
Equation \eqref{eq:lengths} gives
\[
 \beta_u+\beta_v+\beta_w=L-A+B
 \quad(\{u,v,w\}\in E).
\]
Apply Lemma~\ref{lem:kernel} to
\(q_v=\beta_v-(L-A+B)/3\).
All the \(\beta_v\), and hence all the \(\alpha_v\), are equal.

The summands \(x_vh_v\) yield \(\bx_v\bh_v=\br\).
Thus the \(\bx_v\) are prefixes of the same word of the same length,
so they are identical. In particular,
\[
 \bx_s\bh_t=\bx_t\bh_t=\br.
\]
The concatenation is a subword of a member of \(W\), proving
\(\sigma_{\mathbb{H}}\).

For failure in the four test semirings, use the following evaluations:
\[
\begin{array}{c@{\qquad}cccc}
\hline
\text{Semiring}&x_v&y_v&z_v&h_v\\
\hline
S_{\mathbb{H}}&a_v&a_v&a_v&\ell_v\\
S_{\mathbb{H}}^{abc}&a_v&b_v&c_v&r_v\\
S_{\mathbb{H}}^{aac}&a_v&a_v&c_v&r_v\\
S_{\mathbb{H}}^{aba}&a_v&b_v&a_v&r_v\\
\hline
\end{array}
\]
The matching rules make every summand of
\(\bu_{\mathbb{H}}\) equal to \(\omega\). Since \(s\ne t\), the extra
product \(x_sh_t\) is zero. Thus the two sides of
\(\sigma_{\mathbb{H}}\) take the different values \(\omega\) and
\(0\).
\end{proof}

\begin{theorem}\label{thm:general-hereditary-nfb}
If some word of \(W\) has a repeated letter in its strict interior,
then \(S(W)\) is nonfinitely based.
\end{theorem}

\begin{proof}

Choose consecutive internal occurrences of a letter \(a\) in a word
of \(W\). With one adjacent letter on either side, they give a
subword
\[
 \bu=ca\bq ad,
\]
where \(c,d\) are letters and \(\bq\) is a possibly empty word
containing no \(a\).

First suppose \(\bq\) is empty. According to whether \(c\) and \(d\)
equal \(a\), the subword is of one of the forms
\[
 ca^2d\quad(c,d\ne a),\qquad a^3d\quad(d\ne a),\qquad
 ca^3\quad(c\ne a),\qquad a^4.
\]
{For a hyperforest \(\mathbb{F}\), Propositions~\ref{prop:K-forest},
\ref{prop:T-forest}, and~\ref{prop:R-forest} give the corresponding}
representations over \(S(\bu)\). In the case \(ca^3\), reverse the
representation over \(S(a^3c)\).

If \(\bq\) is nonempty and \(\bq=c=d=b\) is a single letter, then
\(\bu=babab\) and Proposition~\ref{prop:J-forest} applies, after
renaming letters. In every other case,
Proposition~\ref{prop:internal-repeat-coordinates} gives
\(S^{abc}_{\mathbb{F}}\in\mathsf{V}(S(\bu))\).

Now fix \(m\geq1\). Choose a {just non-$2$-colourable}
\(3\)-uniform hypergraph \(\mathbb{H}\) of girth greater than \(g_m\),
as in Lemma~\ref{lem:existence}. Use
{\(S_{\mathbb{H}},S_{\mathbb{H}}^{abc},S_{\mathbb{H}}^{aac}\),}
or \(S_{\mathbb{H}}^{aba}\) according to the preceding cases, taking
the opposite algebra for \(ca^3\). Denote the resulting semiring by
\(L_m\).
The forest reduction in Lemma~\ref{lem:local-membership}, together
with the relevant representation, puts every at most \(m\)-generated
subsemiring of \(L_m\) in \(\mathsf{V}(S(\bu))\). Since \(\bu\) is
a subword of a member of \(W\),
Lemma~\ref{lem:general-subdirect} gives
\[
 \mathsf{V}(S(\bu))\subseteq\mathsf{V}(S(W)).
\]
On the other hand, \(L_m\) fails \(\sigma_{\mathbb{H}}\), which holds
in \(S(W)\) by Proposition~\ref{prop:universal}. For an opposite
algebra, use the reversed identity; here $W^{\rev}=\{\bw^{\rev}:\bw\in W\}$, and it holds in every \(S(W)\)
because word reversal identifies \(S(W)^{\op}\) with
\(S(W^{\rev})\).
Thus \(L_m\notin\mathsf{V}(S(W))\), and
Lemma~\ref{lem:local-criterion} proves the theorem.

\end{proof}

\section{Finite basability}
\label{gen:positive-section}

Throughout this section, $W$ is a nonempty family of nonempty words,
every word in $W$ has linear strict interior, and $A=S(W)$. A map between the nonzero parts of two semirings will be
called a \emph{product map} if it preserves every product
that is nonzero in its domain. Such a map need not preserve zero products or addition.
We will represent finite subdirectly irreducible algebras by families
of these maps. Their nonzero elements can be described as labelled
intervals of factorizations of a common root. Dickson's lemma reduces
the existence of the required maps to a bounded-size subalgebra;
a separate argument reduces root reflection to products of at most
seven factors.

In addition to \eqref{gen:M}, the semiring $A$ satisfies
\begin{gather}
 tx^2s\approx0,\label{gen:Z1}\\
 txyxs\approx0,\label{gen:Z2}\\
 vt\preceq xyz+yt+uvw+vs,\tag{L}\label{gen:L}\\
 yx\preceq xy+yz+zt,\tag{Q}\label{gen:Q}\\
 bxy\preceq axy+za+buv+tb,\tag{AP}\label{gen:AP}\\
 bxy\preceq axy+za+uvb+bt,\tag{AM}\label{gen:AM}
\end{gather}
as well as the reversals of (L) and (AP). The zero identities are
understood in their constant-free form, as specified in the
preliminaries. In particular, \eqref{gen:Z1} implies $x^4\approx0$.

Here are the elementary word facts establishing these identities.
{Every subword of a word with linear strict interior again has that
property. A nonempty word that is both a proper prefix and a strict
internal subword of such a word must be its first letter: a longer
word would repeat its second letter at two strict internal positions.
For (L), a common nonzero value $\br=xyz=yt=uvw=vs$ makes both $y$
and $v$ proper prefixes and strict internal subwords of $\br$.
Thus $y=v$ and $vt=yt=\br$. This proves (L), and reversal proves its dual.} The zero identities follow directly
from the two internal occurrences of the first letter of the assigned
value of $x$.

A word with linear strict interior has at most one nonempty proper
border, except for a cubic word $aaa$. Indeed, a border of length at
least three repeats its second letter internally. If both lengths one
and two are borders, the first two letters are equal; this repeats a
letter internally unless the word is $aaa$. For (Q), a common nonzero
value $\br=xy=yz=zt$ makes $y$ and $z$ borders. Outside the exception,
$y=z$ and cancellation gives $x=y=z=t$; in the exception, all words
are unary and $yx=\br$ directly. For (AP) and (AM), the values of $a,b$
are borders of the common root. They coincide outside the exception.
For $\br=aaa$, a border used as the first or last subword of a product
of three nonempty words must be the one-letter word $a$. Hence these
identities also hold when cubic words and longer words both occur in
$W$.

Let $\Phi$ consist of the sixteen flat-variety identities
\[
 x_1(s+t)x_2+y_1sy_2\approx x_1(s+t)x_2+y_1ty_2,
\]
where each of the four outside contexts may independently be omitted.
Relative to the ai-semiring axioms, these define the variety generated
by flat semirings, and every subdirectly irreducible member of that
variety is flat; see~\cite[Lemma~2.1]{jrz}.

\begin{lemma}\label{gen:local-finiteness}
{Every variety of ai-semirings satisfying \eqref{gen:Z1}--\eqref{gen:Z2} is locally}
finite. Every finite member is nilpotent.
\end{lemma}
\begin{proof}
In $m$ generators, a word of length greater than $m+2$ repeats a
generator in its strict interior. Apply \eqref{gen:Z1} if two chosen
occurrences are adjacent, and \eqref{gen:Z2} otherwise. The product is
zero. Terms consequently reduce to the common zero or sums of subsets
of the finitely many words of length at most $m+2$. This proves local
finiteness. Distributing a product of more than $m+2$ arbitrary elements
of an $m$-generated algebra produces only zero generator monomials,
so the algebra is nilpotent as well.
\end{proof}

{For the structural analysis below, $T$ is a finite nontrivial subdirectly irreducible}
model of $\Phi$, \eqref{gen:Z1}--\eqref{gen:Z2}, and the structural
identities above. It is flat and nilpotent. Flatness and distributivity
give $0$-cancellation: $ab=ac\ne0$ implies $b=c$, since otherwise
$a(b+c)=a0=0$ while $ab+ac=ab\ne0$. Right cancellation is analogous.

{Every nonzero element admits a nonzero extension of maximal
factorization length, since $T$ is nilpotent. Such an extension is
annihilated on both sides by every element: a nonzero multiplication
on either side would give a longer extension.} There is only one nonzero
annihilator, denoted by $\omega$. Otherwise collapsing $\{0,a\}$ and
$\{0,b\}$ for distinct nonzero annihilators would give two nontrivial
congruences with equality intersection, contradicting subdirect
irreducibility. Thus every nonzero element divides $\omega$, with
either of its two contexts possibly absent.

Index the proper root factorizations by
\[
 \lambda_i\rho_i=\omega\qquad(i\in N).
\]
$0$-cancellation makes the $\lambda_i$ pairwise distinct, and likewise
the $\rho_i$. Define $i<j$ when $\lambda_i m=\lambda_j$ for some nonzero
$m$. Such $m$ is unique; denote it by $\mu_{ij}$. Cancellation gives
$\mu_{ij}\rho_j=\rho_i$. Associativity makes this relation transitive,
and nilpotence makes it irreflexive: $\lambda_i m=\lambda_i$ would
imply $\lambda_i m^r=\lambda_i\ne0$ for every $r$.

Identity (M) gives uniqueness of strict internal contexts. If
$\lambda m\rho=\lambda' m\rho'=\omega$ with all four contexts nonzero,
then (M) gives $\lambda'm\rho=\omega$. Cancelling the right context
and then $m$ gives $\lambda=\lambda'$; cancelling on the other side
gives $\rho=\rho'$. Consequently each strict internal label occurs on
exactly one interval.

Adjoin endpoints $\bot,\top$ to the poset $N$. Its nonzero interval
labels are
\[
 [\bot,i]:\lambda_i,\qquad [i,\top]:\rho_i,\qquad
 [i,j]:\mu_{ij}\ (i<j),\qquad [\bot,\top]:\omega.
\]
Every nonzero element occurs. Let $P_0,S_0,M_0$ denote the proper-prefix,
proper-suffix, and strict-internal label families, respectively. Labels
are distinct within each family, although a label may occur in more
than one family. Nonzero multiplication is exactly concatenation of
adjacent interval occurrences. One direction follows from associativity
and cancellation of a common root prefix. For the converse, extend a
nonzero product $xy$ to $\ell xyr=\omega$ and split the factorization
at its successive cuts. These cuts give adjacent occurrences of $x,y$.
The interval poset has finite height, bounded by the nilpotence index
of $T$; no common bound over different $T$ is asserted.

Identity (L) implies $|P_0\cap M_0|\le1$. Indeed, if $b,c$ lie in this
intersection, choose factorizations $xby=bt=ucv=cs=\omega$.
Applying (L) gives $ct=\omega=bt$, so $b=c$. The reversed identity gives
$|S_0\cap M_0|\le1$.

{Call a label in $P_0\cap S_0$ a \emph{border label}, and call it
\emph{active} if it is the first or last factor in a factorization of
$\omega$ into three nonzero factors. There is at most one active border
label. Indeed, if $a,b$ are active as first factors, choose
$axy=za=buv=tb=\omega$; (AP) gives $bxy=axy$, hence $a=b$.
The reversed identity handles two active last factors. If $a$ is
active as a first factor and $b$ as a last factor, choose
$axy=za=uvb=bt=\omega$ and apply (AM), again obtaining $a=b$.} There is also at most one decomposable border
label. In fact, if $b=xy$ is a border, extending this product both on
the right and on the left shows that $x\in P_0\cap M_0$ and
$y\in S_0\cap M_0$, and both factors are unique.

We now bound the number of labels with more than one interval
occurrence after discarding redundant components.

The \emph{root graph} has vertices the nonzero labels other than
$\omega$, and $x\to y$ when $xy=\omega$. $0$-cancellation gives
in-degree and out-degree at most one. Identity (Q) says that a walk
$x\to y\to z\to t$ has $y\to x$, forcing $z=x$ and $t=y$.
Every component containing an edge is therefore a loop, a directed
two-cycle, a one-edge directed path, or a two-edge directed path.
It has at most three vertices and at most two border labels.

{An edge-containing component is \emph{pure} if all its vertices are
indecomposable and every nonzero product of at least two factors,
one of which is a vertex of the component, equals $\omega$.} {Since $\omega$ annihilates $T$ on both sides, the product condition
is equivalent to $vx,xv\in\{0,\omega\}$ for every vertex $v$ of
the component and every $x\in T$. Such a component, together with
$0,\omega$, is a $\{0,\omega\}$-direct union summand:} different components have
zero cross-products, and no outside product has a vertex of the pure
component as its value. {Isolated vertices of this ordinary directed root graph will always
be retained; the earlier convention excluding isolated vertices applies
only to hypergraphs.}

\begin{lemma}\label{gen:component-compression}
At most ten labels in the non-pure part have multiple interval
occurrences. For a requirement involving at most $k$ prescribed nonzero
labels, one may delete pure components to obtain a subsemiring $T'$
such that:
\begin{enumerate}[$(1)$]
\item all prescribed labels remain, and every nonzero element of $T'$
divides $\omega$ within $T'$;
\item at most $18+2k$ labels have multiple interval occurrences;
\item every product map on $T'$ extends to $T$, with the
prescribed values unchanged.
\end{enumerate}
\end{lemma}
\begin{proof}
At most four non-pure root-graph components contain border labels.
To verify this, first consider a loop or a two-cycle, whose vertices
are all borders. A decomposable vertex is the unique decomposable
border. A vertex participating in a nonroot nonzero product, after
extension to the root, is active or lies in $P_0\cap M_0$ or
$S_0\cap M_0$.

Now consider a two-edge path $x\to b\to y$. The preceding argument
applies to $b$. If $x$ is decomposable, $xb=\omega$ makes $b$ active.
If $xt=e\ne0,\omega$, extend $e$ to the root. A nonempty left context
puts $x$ in $P_0\cap M_0$. Otherwise there is a nonempty right context
$r$, and $xtr=xb=\omega$ gives $b=tr$, so $b$ is decomposable.
If $tx\ne0,\omega$, extend this product to the root. A nonempty right
context makes $x$ internal; an absent right context makes $x$ a suffix,
which is impossible for the initial vertex of this path. Thus
$x\in P_0\cap M_0$. The endpoint $y$ is handled dually.

Every non-pure component containing a border therefore contains one
of at most four distinguished labels: the active border, the
decomposable border, the element of $P_0\cap M_0$, or the element of
$S_0\cap M_0$, when they exist. There are at most four such components,
and hence at most eight border labels in them. The other two possible
pairwise family intersections give the bound ten.

Retain all non-pure components, all isolated vertices, one representative
of each pure component type present, and all pure components containing
a prescribed label. Delete the other pure components. The remaining
set, with $0,\omega$, is a subsemiring $T'$: a deleted vertex is
indecomposable, and nonzero cross-products are confined to its own
component. Root extensions of retained elements stay available.
Indeed, a deleted label used in a strict internal root factorization
would participate in a nonroot product, contrary to purity; a deleted
complement in a binary root factorization would be in the same root
component as the retained factor.

At most $4+k$ pure components remain, each with at most two border
labels and with no internal labels. Thus at most $10+2(4+k)=18+2k$
labels have multiple occurrences. Finally, use a graph isomorphism
to a retained representative to copy the values of a product map to each deleted component. Its only nonzero products are the
displayed root edges, so this extends the map. All prescribed labels
were retained.
\end{proof}

The next lemma extracts maps to $A$ from finite members of
$\mathsf{V}(A)$. We then show that the required maps on a finite
subdirectly irreducible algebra are determined by subalgebras of
bounded size.

\begin{lemma}\label{gen:diagram-witness}
Let $B\in\mathsf{V}(A)$ be a finite nilpotent flat semiring with a
nonzero element $\omega$ divisible by every nonzero element. There
is a product map $B\setminus\{0\}\to A\setminus\{0\}$.
Moreover, such maps separate distinct nonzero elements, and for any
{specified finite zero product $e_1\cdots e_j=0$ with all $e_i\ne0$, some such map satisfies}
$h(e_1)\cdots h(e_j)\ne h(\omega)$.
\end{lemma}
\begin{proof}
Let $h_B$ bound the length of nonzero products in $B$. Introduce a
variable $x_b$ for each $b\ne0$, and let $\bp_B$ be the sum of all words
$x_{b_1}\cdots x_{b_s}$, $1\leq s\leq h_B$, with source product
$b_1\cdots b_s=\omega$. This is a nonempty finite polynomial containing
$x_\omega$, and every variable occurs. Its canonical value in $B$
is $\omega$. Therefore $\bp_B\approx0$ cannot hold in $A$.

A nonzero evaluation of $\bp_B$ in $A$ assigns the same actual word $\br$
to every summand, and assigns nonzero values to every variable. If
$bc=e\ne0$ in $B$, choose a root extension $\ell et=\omega$, allowing
absent contexts. Both $\ell et$ and $\ell bct$ have length at most
$h_B$ as source factorizations, so their variable words are summands.
Their assigned concatenations equal $\br$. Cancelling the context words
gives $h(b)h(c)=h(e)$ as words, and this product is a subword of $\br$,
so is nonzero in $A$.

For $e\ne f$, replace $e$ by $f$ in a root extension of $e$, obtaining
a variable word $\bq$. $0$-cancellation shows that its canonical value
differs from $\omega$. Thus $\bp_B\approx \bp_B+\bq$ fails in $B$ and hence
in $A$. A witnessing evaluation has $\bp_B\ne0$ and separates $e,f$.
For the specified zero product use $\bq=x_{e_1}\cdots x_{e_j}$ instead.
The same failed absorption identity provides the required map.
There is no restriction on $j$ and no uniform nilpotence assumption
on $A$; only the finite source supplies the bound $h_B$.
\end{proof}

For $0\leq i<j\leq\ell(\br)$, write $\br[i,j]$ for the subword
in positions $i+1,\ldots,j$. For a product map on a root-interval
semiring, put
$\br=h(\omega)$ and $f(i)=|h(\lambda_i)|$, with $f(\bot)=0$ and
$f(\top)=\ell(\br)$. Root factorizations give
\[
 h(\lambda_i)=\br[0,f(i)],\qquad h(\rho_i)=\br[f(i),\ell(\br)],\qquad
 h(\mu_{ij})=\br[f(i),f(j)].
\]
Thus $f$ is a strict order-preserving integer cut map. Conversely,
such a cut map defines a product map precisely when all
interval occurrences of each repeated label receive the same word.
No zero-product condition is imposed.

We recall the required form of Dickson's lemma~\cite{dickson1913}.
Every upward-closed subset of $\mathbb N_0^s$ has finitely many minimal
elements and is their upward closure. For completeness, every infinite
sequence of nonnegative integers has an infinite nondecreasing
subsequence: either a value occurs infinitely often, or values are
unbounded after every finite initial segment and one selects a strictly
increasing subsequence. Apply this successively to the $s$ coordinates.
Every infinite sequence in $\mathbb N_0^s$ then contains a comparable
pair in its original order. An infinite set of minimal elements is
therefore impossible. Each point of an upward-closed set lies above
a minimal element, by minimizing within the finite box below it.

\begin{lemma}\label{lem:dickson-localization}
Let $W$ be any nonempty family of words with linear strict interiors,
and put $A=S(W)$. Fix $k\geq2$. There exists an integer $n(W,k)$ with
the following property. Let $T$ be a finite nontrivial subdirectly irreducible
ai-semiring satisfying $\Phi$ and all the structural and zero identities
above. Suppose also that all at-most-$n(W,k)$-generated subsemirings of $T$ belong to $\mathsf{V}(A)$. Then:
\begin{enumerate}[$(1)$]
\item a product map $T\setminus\{0\}\to A\setminus\{0\}$ exists;
\item for every two distinct nonzero elements, such a map separates them;
\item {for every fixed zero product $e_1\cdots e_j=0$ of nonzero factors, where $j\leq k$,}
there is such a map satisfying
\[
 h(e_1)\cdots h(e_j)\ne h(\omega).
\]
\end{enumerate}
Here a product map is required to preserve products that are
nonzero in the source, but need not preserve zero products.
\end{lemma}

\begin{proof}
Fix one of the three assertions, together with the labels to be
separated or the zero product to be distinguished, when applicable.
Apply Lemma~\ref{gen:component-compression}, retaining one representative
of each pure component type and every component containing a chosen label. Call the resulting
subsemiring $T'$. At most $k$ labels are prescribed. There are at most ten
repeated labels in the non-pure part, and at most two per retained pure
component. Thus there are at most
\[
 b_k=10+2(4+k)=18+2k
\]
repeated labels in $T'$. Each has at most three interval occurrences.
Let $C$ contain every internal endpoint of every such occurrence, so
$|C|\le6b_k=108+12k$. Choose an occurrence for every prescribed label and
include its internal endpoints in $D$, where $|D|\le2k$.
Put $K=C\cup D\cup\{\bot,\top\}$. Then
\[
 |K|\leq m_k:=110+14k.
\]

We record the finite \emph{configuration}: the ordered set induced on $K$,
the two distinguished endpoints, the at most three labelled occurrences
of every repeated label, the selected intervals for the prescribed labels,
their repetitions, and which assertion is under consideration, including
the order of factors in a prescribed product. For fixed $k$
there are only finitely many such configurations up to renaming the nodes
and label indices. All interval endpoints mentioned in this record lie
in $K$.

Fix one such configuration $\pi$. Let $E_\pi$ be the set of its strictly
comparable pairs $(p,q)$, and put $s=|E_\pi|$. A \emph{realization} for
this configuration consists of an actual subword $\br$ of a word in $W$ and
prescribed integer cuts $g:K\to\{0,\ldots,\ell(\br)\}$, with
$g(\bot)=0$, $g(\top)=\ell(\br)$, such that:
\begin{enumerate}[$(1)$]
\item $g(p)<g(q)$ whenever $p<q$ in $K$;
\item the interval words $\br[g(p),g(q)]$ assigned to all occurrences of
the same repeated label coincide;
\item the assigned intervals satisfy the prescribed condition, if a separation
or a zero-product distinction has been specified.
\end{enumerate}
For existence there is no third condition. For separation, condition
(3) is inequality of the two selected interval words. For a prescribed zero product it says that the concatenation of the selected interval words is
not $\br$. The latter is equivalent to the corresponding product in $A$
being different from $\br$, since equality to $\br$ is possible only when
the concatenation itself is $\br$.

To each realization associate its capacity vector
\[
 c(g)_{pq}=g(q)-g(p)\in\mathbb N_0\qquad((p,q)\in E_\pi).
\]
Let $\mathcal D_\pi\subseteq\mathbb N_0^s$ be the componentwise downward
closure of all these vectors. This set depends only on $W$ and the
finite configuration, not on the remaining source poset. Its complement
is an upward-closed subset of $\mathbb N_0^s$, so Dickson's lemma gives
a finite set $\mathcal B_\pi$ of minimal excluded vectors. This includes
the possibilities $\mathcal B_\pi=\varnothing$ and, when there are no
realizations, $\mathcal B_\pi=\{\mathbf0\}$.

Let $P$ be the full interval poset of $T'$, and define
\[
 \delta(P)_{pq}=\operatorname{dist}_{P}(p,q),
\]
{where $\operatorname{dist}_P(p,q)$ is the largest number of
strict steps in a chain from $p$ to $q$, with
$\operatorname{dist}_P(p,p)=0$.}
A realization on $K$ extends to a strict integer cut map on $P$ exactly
when $c(g)\ge\delta(P)$ componentwise. This is the standard extension
criterion: the extension is
\[
 f(z)=\max_{p\in K,\ p\leq z}
       \bigl(g(p)+\operatorname{dist}_P(p,z)\bigr).
\]
{For $z\in K$, the capacity inequalities give
$g(p)+\operatorname{dist}_P(p,z)\leq g(z)$ for every $p\in K$ with
$p\leq z$, and the term $p=z$ gives equality. If $z<z'$, extending
a maximizing chain by $z'$ shows that $f(z')\geq f(z)+1$.
The values at $\bot,\top$ therefore bound the range, so this is the
required extension.} Such an extension gives exactly a
product map satisfying this condition. Consequently, if no global map satisfying this condition exists on $T'$, then $\delta(P)\notin\mathcal D_\pi$.
Choose $b\in\mathcal B_\pi$ with $b\le\delta(P)$.

For each $(p,q)$ with $b_{pq}>0$, select a chain from $p$ to $q$ having
exactly $b_{pq}$ strict steps. This is obtained by taking a subsequence,
including both endpoints, of a longest chain. Let $U$ contain $K$ and
all points of these chains. Then
\[
 |U|\leq |K|+\sum_{(p,q)\in E_\pi}b_{pq},
 \qquad \delta(P|_U)\geq b.
\]
Coordinates with $b_{pq}=0$ need no chain. The construction also works
when $b=\mathbf0$, in which case $U=K$.

Let $T'_U$ consist of $0$ and every label of an interval with both
endpoints in $U$. It is a subsemiring: every occurrence of a repeated
label has both endpoints in $C\cup\{\bot,\top\}\subseteq U$, while
an unrepeated label has only its single occurrence. Every nonzero
product of selected labels is therefore witnessed by adjacent selected
intervals, whose combined interval is also selected. All complementary
root intervals are selected as well. The interval poset of $T'_U$ is
exactly the induced poset $P|_U$: any root cut in $T'_U$ comes from such
adjacent occurrences, so cannot introduce an unselected cut.
In particular, every nonzero label of $T'_U$ divides $\omega$.

There are finitely many configurations and finitely many excluded vectors for
each. Thus there is a finite number $q(W,k)\geq m_k$ bounding all the displayed
sizes of $U$ whenever an obstruction occurs. Take
\[
 n(W,k)\ge\max\{8,q(W,k)^2+1\}.
\]
There are at most $|U|^2$ labels on the selected intervals, so $T'_U$
has at most $n(W,k)$ elements. It belongs to $\mathsf{V}(A)$ by the
hypothesis of the lemma.

Apply Lemma~\ref{gen:diagram-witness} to $T'_U$. It supplies a product map satisfying the same condition. Its actual root word and its cuts on $K$
give a realization for $\pi$. Its capacity vector is at least
$\delta(P|_U)\geq b$. This contradicts $b\notin\mathcal D_\pi$.
Hence a map satisfying the prescribed condition exists on $T'$. Copying its values on a retained pure
component representative to each deleted component of the same type
extends it to all of $T$. There are no nonzero cross-products to check,
and every prescribed label was retained. This proves all three assertions.
\end{proof}

\begin{lemma}\label{lem:seven-factor-reflection}
{Let $T$ be a finite nilpotent flat semiring with a nonzero root $\omega$
divisible by all its nonzero elements. Suppose that $T$ admits the
labelled root-interval description above and that its intervals have}
unique labels within each of the prefix, suffix and strict-internal
families. Let $A=S(W)$, where all words in $W$ have linear strict
interior. Suppose that the nonempty family $\mathcal H$ of all maps
$h:T\setminus\{0\}\to A\setminus\{0\}$ preserving nonzero products
separates distinct nonzero elements. If for every zero source product of
at most seven nonzero factors some $h\in\mathcal H$ sends that product
to a value different from $h(\omega)$, then the same assertion holds for
zero source products of arbitrary length.
\end{lemma}
\begin{proof}
Suppose $e_1\cdots e_t=0$ and nevertheless
$h(e_1)\cdots h(e_t)=h(\omega)$ for every $h\in\mathcal H$. The assertion
is immediate if $t\le7$, so assume $t>7$. None of the $e_i$ is $\omega$:
a product map sends $\omega$ to the whole root word, which
cannot be a factor in a longer nonempty concatenation equal to itself.

Call an internal index $i\in\{2,\ldots,t-1\}$ exceptional if $e_i$ has
no strict-internal source occurrence. Every such label is a proper
prefix or a proper suffix label. Fix $h_0\in\mathcal H$ and write
$\br=h_0(\omega)$. The factorization of $\br$ supplied by the displayed
equality puts $h_0(e_i)$ strictly internally. If $e_i$ is a prefix label,
its image is also a proper prefix of $\br$. A nonempty word that is both
a proper prefix and a strict-internal subword of a linear-interior word
must be its first letter: otherwise its second letter repeats strictly
internally. This first letter has at most one strict-internal occurrence.
Thus at most one exceptional index can carry a prefix label. The dual
argument allows at most one carrying a suffix label. There are at most
two exceptional indices altogether, including when $\br=aaa$.

For each nonexceptional internal index choose nonzero root contexts
$\lambda_i e_i\rho_i=\omega$. For every $h\in\mathcal H$, the target
word $h(e_i)$ has a unique strict-internal occurrence, since the strict
interior of $h(\omega)$ is linear. Hence the root factorization of the
source context and the root factorization $h(e_1)\cdots h(e_t)$ locate
this occurrence at the same place. If $i,i+1$ are both nonexceptional,
comparison of the prefixes ending at their common boundary gives
\[
 h(\lambda_i e_i)=h(\lambda_{i+1})
 \qquad\text{for every }h\in\mathcal H.
\]
Separation yields $\lambda_i e_i=\lambda_{i+1}$ in $T$. Consequently a
maximal run $e_i\cdots e_j$ of nonexceptional internal factors has
nonzero source product: successively substitute these equalities into
$\lambda_j e_j\rho_j=\omega$.

Keep $e_1,e_t$ and the at most two exceptional factors as separate
blocks, and replace each of the at most three intervening runs by its
nonzero source product. This produces at most seven nonzero factors
whose source product is still zero. Every product map preserves
each compressed block, so the compressed word still maps to the root
under every such map. This contradicts the assumed seven-factor
reflection property.
\end{proof}

\begin{theorem}\label{gen:positive}
If every word in a nonempty family $W$ has linear strict interior, then
$S(W)$ is finitely based.
\end{theorem}
\begin{proof}
Choose $n\ge8$ as in Lemma~\ref{lem:dickson-localization} for $k=7$. Let $\mathcal C$
{be the variety of ai-semirings defined by all identities of $A$
involving at most $n$ variables.} Its defining identities include \eqref{gen:Z1}--\eqref{gen:Z2}, the finite flat-variety axioms, and all the
structural identities (M), (L), (Q), (AP), (AM) and the required reversals.
Thus it is locally finite. Every finite nontrivial subdirectly irreducible member
$T$ is flat, $0$-cancellative, and nilpotent, has a unique nonzero
annihilator $\omega$, and all its nonzero elements divide $\omega$.
The one-element algebra already belongs to $\mathsf{V}(A)$.

Every at-most-$n$-generated subsemiring $B$ of $T$ lies in
$\mathsf{V}(A)$. Otherwise take an identity of $A$ failing in $B$,
express the assigned elements by terms in fixed generators of $B$,
and substitute these terms into that identity. The result is an
identity of $A$ in at most $n$ variables failing in $T$, a contradiction.
Lemma~\ref{lem:dickson-localization} therefore supplies product maps from
$T\setminus\{0\}$ to $A\setminus\{0\}$ that exist, separate all
distinct elements, and distinguish all zero source products of at
most seven factors from the root. Take all product maps as
coordinates. Lemma~\ref{lem:seven-factor-reflection} excludes a zero source product
of any length from having the root value in every coordinate. A nonzero
source product different from $\omega$ is distinguished from the root
by element separation. Full root reflection therefore holds, and
Lemma~\ref{lem:representation} gives $T\in\mathsf{V}(A)$.

Every finite algebra in $\mathcal C$ is a subdirect product of finite
subdirectly irreducible quotients, so belongs to $\mathsf{V}(A)$.
Local finiteness then gives $\mathcal C=\mathsf{V}(A)$: a failed identity
would already fail in a finitely generated subalgebra.

Finally, modulo the ai-semiring axioms and \eqref{gen:Z1}--\eqref{gen:Z2}, every term in $n$
variables is equivalent either to the common zero term $x_1^4$ or to
a nonempty sum of distinct variable words of length at most $n+2$.
There are finitely many such normal forms. Take all identities valid
in $A$ between pairs of these normal forms, together with the
ai-semiring axioms and the finitely many constant-free expansions of \eqref{gen:Z1}--\eqref{gen:Z2}.
They entail every valid identity of $A$ in at most $n$ variables,
and therefore axiomatize $\mathsf{V}(A)$. This is a finite basis.
\end{proof}

\begin{proof}[Proof of Theorem~\ref{thm:main}]
Lemma~\ref{lem:general-linear-interior-identities} proves the equivalence
of conditions (2), (3), and (4). Theorem~\ref{thm:general-hereditary-nfb}
gives (1)$\Rightarrow$(2), and Theorem~\ref{gen:positive} gives
(2)$\Rightarrow$(1).
\end{proof}

\end{document}